\RequirePackage{fix-cm}
\documentclass[draft,numbook]{svjour3}                     
\smartqed  
\usepackage{graphicx}
\usepackage{amssymb,amsmath}
\usepackage{verbatim}
\usepackage{color}

\spnewtheorem{hypothesis}{Hypothesis}[section]{\bf}{\it}

\spnewtheorem*{proofofresult}{Proof}{\it}{\rm}

\journalname{Journal of Evolution Equations}
\begin{document}

\title{The essential spectrum of  periodically-stationary solutions of the complex 
Ginzburg-Landau equation
\thanks{J.Z. was supported by the NSF under grant DMS 1620293 and thanks the Department of Mathematics at  UNC Chapel Hill for hosting his Fall 2018 sabbatical, during which this work began.}
\thanks{Y.L. was supported by the NSF under grant DMS 1710989 and thanks the Courant Institute
of Mathematical Sciences and, especially, Prof. Lai-Sang Young, for the opportunity to visit
the Institute where this work was conducted.}
\thanks{J.L.M. was supported in part by NSF CAREER Grant DMS 1352353 and NSF Grant DMS 1909035.}
\thanks{C.K.R.T.J. was supported by the US Office of Naval Research under grant  N00014-18-1-2204.}}

\titlerunning{Essential spectrum of periodically-stationary Ginzburg-Landau pulses}        

\author{John Zweck \and
            Yuri Latushkin \and 
             Jeremy L. Marzuola \and 
             Christopher K.R.T. Jones
             }

\authorrunning{Jones, Latushkin, Marzuola, and Zweck} 

\institute{               John Zweck \at
                Department of Mathematical Sciences, The University of Texas at Dallas, Richardson, TX 75080, USA\\
                \email{zweck@utdallas.edu}
            \and
            Yuri Latushkin \at
             Department of Mathematics, The University of Missouri, 
             Columbia, MO 65211, USA\\
             \email{latushkiny@missouri.edu}
             \and
             Jeremy L. Marzuola \at
             Department of Mathematics, The University of North Carolina at Chapel Hill, 
                NC 27599, USA\\
                \email{marzuola@email.unc.edu}
                \and
                Christopher K.R.T. Jones \at
               Department of Mathematics, The University of North Carolina at Chapel Hill, 
                NC 27599, USA\\
              \email{ckrtj@email.unc.edu}
 }

\date{Received: date / Accepted: date}

\dedication{To Matthias Hieber with best wishes}

\maketitle

\begin{abstract}
We establish the existence and regularity properties of a monodromy operator
for the linearization of the cubic-quintic complex Ginzburg-Landau equation
about a periodically-stationary (breather) solution. We derive
a formula for the essential spectrum of the
monodromy operator in terms of that of the associated 
asymptotic linear differential operator.
This result is obtained using the theory of 
analytic semigroups under the assumption that  the  Ginzburg-Landau equation includes a 
spectral filtering (diffusion) term. We discuss applications to the stability of 
periodically-stationary pulses in ultrafast fiber lasers.
%
\end{abstract}

\section{Introduction}

The cubic-quintic complex Ginzburg-Landau equation (CQ-CGLE)
is a fundamental model for nonlinear waves and coherent structures
that arise in fields such as  nonlinear optics and condensed matter physics~\cite{akhmediev2008dissipative,RMP74p99}. 
The CQ-CGLE supports a wide variety of  solutions, including 
stationary pulses,  periodically-stationary pulses,
fronts, exploding solitons, and chaotic solutions~\cite{akhmediev2008dissipative}. 
While stationary pulses  (solitons) maintain their shape,
periodically-stationary pulses (breathers) change shape as they propagate,
returning to the same shape periodically.

Periodically-stationary pulses are the solutions of primary interest
to engineers designing ultrafast fiber lasers. These lasers 
generate pulses with widths in the picosecond to femtosecond range
and have applications to time and frequency metrology~\cite{chong2015ultrafast,kartner2004few,SIREV48p629}.
Since the advent of the soliton laser~\cite{Duling:91,Mollenauer:84}, researchers have invented several generations of short-pulse, high-energy fiber lasers for
a variety of applications.
In the mid 1990's stretched-pulse (dispersion-managed) lasers
were devised to generate pulses with higher energy and
shorter duration than can be acheived with soliton lasers~\cite{tamura199377,tamura1994soliton}.
Dissipative soliton lasers, in which effects
such as spectral filtering play a significant role, were
introduced in about 2005 and are suitable for high energy applications~\cite{chong2006all,grelu2012dissipative}.
Similariton lasers were introduced in 2010 to create
femto-second pulses with a high tolerance to noise~\cite{chong2015ultrafast,fermann2000self,hartl2007cavity} by
exploiting the theoretical discovery of exponentially growing, self-similar pulses
in optical fiber amplifiers.
The most recent invention is the  Mamyshev oscillator which can  produce
pulses with a peak power in the Megawatt range~\cite{regelskis2015ytterbium,sidorenko2018self}.


The key issue for mathematical modeling of ultrafast lasers is
to determine those regions in the design parameter space
in which stable pulses exist, and within that space to optimize the
pulse parameters.  A significant challenge  
is that from one generation of lasers to the next there has been
an  increase in the amount by which the pulse changes within each 
round trip  (period).
Consequently, soliton perturbation theory~\cite{OptLett11p665,PLA42p5689}, 
which was developed to analyze the stability of stationary pulses, is no longer applicable.
Although low-dimensional reduced ODE models~\cite{tsoy2005bifurcations,tsoy2006dynamical}
and Monte Carlo simulations that use full PDE models~\cite{zweck2018computation}
have both been employed to assess pulse stability, 
to date there is no mathematical theory
to determine  the stability of periodically-stationary laser pulses. 
In ultrafast lasers, the different physical effects (dispersion, nonlinearity, spectral filtering,
saturable gain and loss) occur in  different devices within the laser. For a 
quantitative model it is therefore necessary to use am equation
such as the CQ-CGLE in which the coefficients are piecewise constant 
in the evolution variable. However, constant coefficient models are also often used to
gain qualitative insight into the system behavior~\cite{renninger2008dissipative}.

In this paper, we take a first step in the development of  a stability  theory
for ultrafast lasers
by calculating the essential spectrum of the monodromy operator
of the linearization of the CQ-CGLE about a periodically-stationary pulse solution. Because variable coefficient equations  are more challenging to analyze,
we restrict attention to the constant coefficient CQ-CGLE,
which is a phenomenological, distributed model for short-pulse fiber
lasers~\cite{grelu2012dissipative,SIREV48p629}.
This equation has two important classes of solutions that are periodic
in the temporal variable. 

The first class is the family of  Kuznetsov-Ma (KM)
breathers~\cite{kuznetsov1977solitons,ma1979perturbed},
which are analytical solutions of the focusing nonlinear Schr\"odinger  equation (FNLSE),
a special case of the CQ-CGLE. These solutions, which were discovered using 
integrable systems techniques, have a non-zero background at spatial infinity.
A numerical Floquet spectrum computation 
by Cuevas-Maraver \emph{et al.}~\cite{cuevas2017floquet}
suggests that the KM breather is linearly unstable.
Mu\~noz~\cite{munoz2017instability} recently used a Lyapunov functional
to prove that because of their non-zero background, these breathers 
are unstable under small $H^s$ ($s>\frac 12$) perturbations.
Integrable systems techniques 
have also been used to find breather solutions of the modified and
higher-order KdV equations and the Gardner hierarchy~\cite{clarke2000chaos,alejo2019dynamics}.
In a recent series of papers~\cite{alejo2018nonlinear,alejo2013nonlinear,alejo2017variational},
Alejo exploited the integrability structure
of these PDE's and Lyapunov functional techniques to establish the nonlinear
stability of several  such breathers.

The second class consists of the 
periodically-stationary pulses discovered numerically by Akhmediev and his
collaborators~\cite{akhmediev2001pulsating,tsoy2005bifurcations,tsoy2006dynamical}.
These solutions were found in the case that the CQ-CGLE includes a spectral filtering term. 
Although Akhmediev \emph{et al.} provided strong numerical evidence for the existence of these solutions, there are no known analytic formulae for these solutions and  no mathematical proof of their existence.
However, using numerical simulations and reduced ODE models,
they provided numerical evidence for the existence of both stable and unstable  
periodically-stationary pulses.

In Floquet theory, the stability of periodic solutions of a system of ODE's is
characterized by the spectrum of the monodromy matrix of the
linearization of the system about the  solution.
Although Floquet methods have been developed for solutions of PDE's that are
periodic in the spatial variables (see for 
example~\cite{gesztesy1995floquet,jones2010stability,kuchment2012floquet,reed1980methods}), we are only aware of a few results for solutions that  are periodic in the temporal (evolution) variable. 
Wilkening~\cite{wilkeningharmonic} developed a numerical method  to study the stability of standing water waves and other time-periodic solutions of the free-surface Euler equations.
Motivated by problems from quantum mechanics,
Korotyaev~\cite{korotyaev1983spectrum,korotyaev1985eigenfunctions}  studied  Schr\"odinger operators that have
a real scalar potential which is periodic in time and rapidly decaying in space.
He showed that the  monodromy operator has no  
singular continuous spectrum. However, this result relies heavily on fact that the  evolution operator is unitary, which is not
the case when the CQ-CGLE includes a spectral filtering term.
Similar results are discussed in the book of Kuchment~\cite{kuchment2012floquet}.
Finally, Sandstede and Scheel~\cite{doelman2009dynamics,sandstede2001structure,sandstede2002stability}
have an extensive body of theoretical and numerical results on the stability of  
time-periodic perturbations of spatially-periodic traveling waves.
However, because of the underlying spatial periodicity in their formulation, these results are not applicable to laser systems.

The results in this paper can be summarized as follows. 
In Section~\ref{Sec:Exs}, we review the periodically-stationary solutions
of Kuznetsov-Ma and Akhmediev and define the time-periodic operator,
$\mathcal L(t)$, which is obtained by linearization of the CQ-CGLE about
a periodically-stationary solution. We also discuss the asymptotic operator,
$\mathcal L_\infty$, associated with $\mathcal L(t)$. 
In Section~\ref{Sec:EssSpecL}, we calculate the 
essential spectrum of $\mathcal L_\infty$ with the aid of some results
from the text of Kapitula and Promislow~\cite{kapitula2013spectral}, and 
we show that $\mathcal L(t)$ is a relatively compact perturbation of  $\mathcal L_\infty$.
In Section~\ref{Sec:EvolFamily}, 
we establish the existence of an evolution family  for $\mathcal L(t)$ by applying classical results on solutions of initial-value problems for non-autonomous linear differential equations in Banach spaces~\cite{pazy2012semigroups}.
In Section~\ref{Sec:EssSpecM}, we use the results obtained in Section~\ref{Sec:EvolFamily} to define the monodromy operator, $\mathcal M(s)$,
and establish the main result of the paper, which is a formula for the
essential spectrum of $\mathcal M(s)$ in terms of the essential spectrum
of the asymptotic operator, $\mathcal L_\infty$. 
To obtain this result, we assumed that the CQ-CGLE includes a spectral filtering term,
which ensures that the semigroup of the asymptotic operator is analytic. 
Since all fiber lasers have bandlimited gain, this assumption 
holds in applications.
Based on the numerical results of Cuevas-Maraver \emph{et al.}~\cite{cuevas2017floquet}, we conjecture that the formula for the essential spectrum of 
$\mathcal M(s)$ also holds for the KM breather. However a different approach 
will be required to prove such a result, since 
in this case the asymptotic operator is not analytic. 

\section{Motivating examples}\label{Sec:Exs}

We consider  a class of one-dimensional, constant-coefficient
nonlinear Schr\"odinger equations of the form
\begin{equation}
i\partial_t \psi \,\,+\,\, \tfrac 12 \partial_x^2 \psi \,\,+\,\, f( | \psi |^2)\psi=0,
\label{eq:nls}
\end{equation}
where $f$ is a polynomial with complex coefficients. 
We call $t$ the \emph{temporal} or \emph{evolution} variable and  $x$ 
the \emph{spatial} variable. 
We consider solutions, $\psi$, for which 
$f( | \psi |^2)\psi \to 0$ at an exponential rate as $x\to \pm\infty$.
We say that $\psi$ is a \emph{periodically-stationary solution} of \eqref{eq:nls}
if there is a period, $T$, so that $\psi(t+T,x) = \psi(t,x)$ for all $t$ and $x$. 
Our analysis is motivated by two important examples.

\begin{example}\label{ex:KM}
Kuznetsov~\cite{kuznetsov1977solitons} and Ma~\cite{ma1979perturbed}
independently discovered a  family of periodically-stationary solutions
of the FNLSE,
\begin{equation}
i\partial_t \psi \,\,+\,\, \tfrac 12 \partial_x^2 \psi \,\,+\,\, ( | \psi |^2-\nu_0^2)\psi=0,
\label{eq:KMnls}
\end{equation}
with the property that
\begin{equation}
\lim\limits_{x\to\pm\infty} \psi(t,x) = \nu_0.
\end{equation}
Here the parameter, $\nu_0>0$, is the background amplitude.
The Kuznetsov-Ma (KM) breathers are defined in terms of a 
second parameter, $\nu > \nu_0$,  by~\cite{garnier2012inverse}
\begin{equation}
\psi_{\rm{KM}}(t,x)\,\,=\,\, \nu_0 \,\,+\,\, 2\eta\,
\frac{\eta\cos(2\nu\eta t) + i\nu \sin(2\nu\eta t) }{\nu_0 \cos(2\nu\eta t) - \nu \cosh(2\eta x)},
\label{eq:KM}
\end{equation}
where $\eta = \sqrt{\nu^2-\nu_0^2} > 0$. The period of $\psi_{\rm{KM}}$ is 
$T= {\pi}/{\nu\eta}$. 
We observe that $f(|\psi_{\rm{KM}}|^2)\psi_{\rm{KM}} =  (|\psi_{\rm{KM}}|^2 -\nu_0^2)\psi_{\rm{KM}}\to 0$ at an exponential rate as $x\to\pm\infty$.
\end{example}

\begin{example}\label{ex:GL}
The CQ-CGLE~\cite{akhmediev2008dissipative,RMP74p99}  is given by
\begin{equation}
i\partial_t\psi + (\tfrac D2 - i \beta)\partial_x^2 \psi -  i\delta\psi +  
\left[\gamma-i\epsilon + (\nu-i\mu)  |\psi|^2\right]|\psi|^2\psi
\,\,=\,\, 0.
\label{eq:GL}
\end{equation}
Among other applications, the CQ-CGLE  provides a qualitative model
for the generation of short-pulses in mode-locked fiber lasers~\cite{akhmediev2008dissipative,SIREV48p629}. 
In this context, the parameters in the equation can  be interpreted as follows.
The coefficient, $D$, is the fiber dispersion, which is positive in the 
anomalous or focusing dispersion regime and negative 
in the normal or defocusing dispersion regime.
Spectral filtering  is modeling using the term with  coefficient $\beta > 0$.
The remaining terms model  linear gain or loss ($\delta$), 
saturable nonlinear gain ($\epsilon>0$ and $\mu<0$),
and the cubic and quintic nonlinear electric susceptibility of the optical fiber
($\gamma>0$ and $\nu>0$). 

Using a numerical partial differential equation solver, 
Akhmediev and his collaborators~\cite{akhmediev2008dissipative,akhmediev2001pulsating,tsoy2005bifurcations,tsoy2006dynamical} provide strong evidence for the existence of 
periodically-stationary solutions of \eqref{eq:GL}, which they refer to as 
pulsating solitons. 
There are no known analytical formulae for these solutions.
However, the numerical results show that these solutions decay at an exponential rate as $x\to\pm\infty$.
\end{example}

Equations~\eqref{eq:KMnls} and \eqref{eq:GL} are both special cases of the
general nonlinear wave equation
\begin{equation}
i\partial_t\psi + (\tfrac D2 - i \beta)\partial_x^2 \psi + (\alpha -  i\delta)\psi +  
\left[\gamma-i\epsilon + (\nu-i\mu)  |\psi|^2\right]|\psi|^2\psi
\,\,=\,\, 0.
\label{eq:NWE}
\end{equation}
Throughout this paper, we assume that $(D,\beta) \neq (0,0)$. 
Since the linearization of \eqref{eq:NWE} about a solution involves both the linearized unknown and its complex conjugate, we reformulate \eqref{eq:NWE} 
as the system of equations for 
$\boldsymbol\psi = \begin{bmatrix} \Re(\psi) & \Im(\psi)\end{bmatrix}^T$
  given by
\begin{equation}
\partial_t  \boldsymbol\psi  \,\,=\,\,
\left(\mathbf{B} \, \partial^2_x +  \mathbf{N}_0  +
\mathbf{N}_1 |  \boldsymbol\psi |^2  +
\mathbf{N}_2 |  \boldsymbol\psi |^4  \right)
 \boldsymbol\psi,
\label{eq:NWEsys}
\end{equation}
where 
\begin{equation}
 \mathbf{B} = \begin{bmatrix} \beta & -\frac D2 \\ \frac D2 & \beta\end{bmatrix},
\end{equation}
and 
\begin{equation}
\mathbf{N}_0 =  \begin{bmatrix} 
\delta & -\alpha \\ \alpha & \delta
\end{bmatrix}, \qquad
\mathbf{N}_1 =  \begin{bmatrix} 
\epsilon & -\gamma \\ \gamma & \epsilon
\end{bmatrix}, \qquad
\mathbf{N}_2 =  \begin{bmatrix} 
\mu & -\nu \\ \nu & \mu
\end{bmatrix}.
\end{equation}
The linearization of \eqref{eq:NWEsys} about a solution, $\boldsymbol\psi$,
is of the form 
\begin{equation}
\partial_t \mathbf p \,\,=\,\,\mathcal L(t) \mathbf p,
\label{eq:Linearization}
\end{equation}
where $\mathcal L=\mathcal L(t)$ is a second-order, linear differential operator in $x$ 
with real, $t$- and $x$-dependent, matrix-valued
coefficients. 
 If the solution $\boldsymbol\psi$ is periodically stationary
with period $T$, then $\mathcal L$ is periodic in $t$ with 
$\mathcal L(t+T) = \mathcal L(t)$.
Substituting $\boldsymbol\psi_\varepsilon = \boldsymbol\psi+ \varepsilon \mathbf p$
into \eqref{eq:NWEsys} are keeping only terms of order $ \varepsilon$ we find that
\begin{equation}
\mathcal L(t) \,\,=\,\, 
\mathbf{B} \, \partial^2_x +  \widetilde{ \mathbf{M}}(t),
\label{eq:OpL}
\end{equation}
where $\widetilde{\mathbf{M}}(t)$ is the operator of multiplication by
\begin{equation}
\widetilde{ \mathbf{M}}(t,x) = \mathbf{N}_0  +
\mathbf{N}_1 |  \boldsymbol\psi |^2  +
\mathbf{N}_2 |  \boldsymbol\psi |^4 +
\left(2\mathbf{N}_1 + 4\mathbf{N}_2 |\boldsymbol\psi |^2\right)  \boldsymbol\psi \boldsymbol\psi^T.
\label{eq:OpM}
\end{equation}

\section{Essential Spectrum of the Linearized Differential Operator}\label{Sec:EssSpecL}

In this section we introduce the asymptotic operator, $\mathcal L_\infty$, associated with 
the differential operator, $\mathcal L(t)$, 
and determine the  essential spectrum of 
$\mathcal L_\infty$. We also prove that  $\mathcal L(t)$ is a relatively compact perturbation
of $\mathcal L_\infty$. 
Our results rely on a  general theory summarized by 
 Kapitula and Promislow~\cite{kapitula2013spectral} and on a classical compactness theorem for $L^2$ due to Kolmogorov and Riesz~\cite{hanche2010kolmogorov,reed1980methods}.
We begin with the following assumption.

\begin{hypothesis}\label{hyp:EA}
The solution, $\boldsymbol\psi$, about which \eqref{eq:NWEsys} is linearized is assumed to be periodically stationary with period, $T$, and has the following properties:
\begin{enumerate}
\item  For each $t\in[0,T]$, the function $\boldsymbol\psi(t,\cdot) \in L^\infty(\mathbb R,\mathbb C^2)$;
\item For each $t\in[0,T]$, the weak derivative $\boldsymbol\psi_x(t,\cdot) \in L^\infty(\mathbb R,\mathbb C^2)$;
\item There exist constants $r>0$ and $\boldsymbol\psi_\infty \in \mathbb C^2$ 
so that 
\begin{equation*}
\lim\limits_{|x|\to\infty} e^{r|x|} 
\|\boldsymbol\psi(t,x) - \boldsymbol\psi_\infty\|_{\mathbb C^2}
 \,\,=\,\,0, \qquad \text{for all } t\in[0,T].
\end{equation*}
\end{enumerate}
\end{hypothesis}

\begin{remark}
The KM breather \eqref{eq:KM}  satisfies Hypothesis~\ref{hyp:EA}. 
Although we do not have definitive proof, it is reasonable to assume that
the  numerically computed periodically-stationary
solultions discussed in Example~\ref{ex:GL} also satisfy Hypothesis~\ref{hyp:EA}.
\end{remark}

{
\begin{remark}\label{remark:spaces}
Hypothesis~\ref{hyp:EA}  guarantees that for each $t$,
the multiplication operator,  $\widetilde{\mathbf M}(t)$, in~\eqref{eq:OpM} is a bounded
operator on $L^2(\mathbb R, \mathbb C^2)$. 
Since $(D,\beta) \neq (0,0)$, it follows from  \cite[Lemma 3.1.2]{kapitula2013spectral}
that the operator 
$\mathcal L(t) : H^2(\mathbb R, \mathbb C^2) \subset 
 L^2(\mathbb R, \mathbb C^2) \to L^2(\mathbb R, \mathbb C^2)$ is closed.
\end{remark}
}

Next, we have the following proposition, whose proof is self-evident. 

\begin{proposition}
Assume that Hypothesis~\ref{hyp:EA} is met. Then
\begin{equation}
 \mathbf{M}_\infty := \lim\limits_{|x|\to\infty} \widetilde{ \mathbf{M}}(t,x)
\end{equation}
exists and is $t$-independent. Furthermore, the differential operator, 
$\mathcal L$, is exponentially asymptotic in that the leading coefficient, $ \mathbf{B}$, 
is constant and there exists a constant $r>0$ so that 
\begin{equation}
\lim\limits_{|x|\to\infty} e^{r|x|} \,
\left\|\widetilde{\mathbf M}(t,x) - \mathbf{M}_\infty\right\|_{\mathbb C^{2\times 2}} 
\,\,=\,\,0
\qquad \text{for all } t \in [0,T],
\end{equation}
 where   $\| \cdot \|_{\mathbb C^{2\times 2}}$ denotes the matrix norm induced 
 from the Euclidean norm $\| \cdot \|_{\mathbb C^2}$.
\end{proposition}

\begin{definition}
The \emph{asymptotic differential operator}, $\mathcal L_\infty$, associated with 
the exponentially asymptotic operator, $\mathcal L(t)$,
is the $t$-independent operator with constant, matrix-valued coefficients given by
\begin{equation}
\mathcal L_\infty \,\,:=\,\,  \mathbf{B} \,\partial^2_x + 
 \mathbf{M}_\infty.
 \label{eq:OpLinf}
\end{equation}
Just as for the operator, $\mathcal L(t)$, we regard $\mathcal L_\infty$ 
as an operator on $L^2(\mathbb R, \mathbb C^2)$ with domain
$H^2(\mathbb R, \mathbb C^2)$.
Furthermore, if we define $ \mathbf{M}(t,x) := \widetilde{ \mathbf{M}}(t,x) -  
\mathbf{M}_\infty$, then
\begin{equation}
\mathcal L(t) = \mathcal L_\infty +  \mathbf{M}(t),
\end{equation}
where $\mathbf{M}(t,x)\to \mathbf 0$ as $|x|\to\infty$.
\end{definition}

We now review the definition of the essential spectrum we use in this paper. 

\begin{definition}\label{SpecDefs}
Let $X$ be a Banach space and let $\mathcal B(X)$ denote the space of bounded linear operators on $X$. 
Let $\mathcal L: D(\mathcal L)\subset X\to X$ be a closed
linear operator with domain $D(\mathcal L)$ that is dense in $X$. 
The \emph{resolvent set} of $\mathcal L$ is 
\begin{equation}
\rho(\mathcal L) \,\,:=\,\, \{ \lambda \in \mathbb C \, |\, 
\mathcal L - \lambda \text{ is invertible and } (\mathcal L - \lambda)^{-1} 
\in \mathcal B(X) \},
\end{equation}
and for each $\lambda \in \rho( \mathcal L)$, the \emph{resolvent operator} is 
$R(\lambda : \mathcal L) := (\mathcal L - \lambda)^{-1}$. The \emph{spectrum} of $\mathcal L$ is $\sigma(\mathcal L) \,\,:=\,\, \mathbb C \setminus \rho(\mathcal L)$.
The \emph{point spectrum} of $\mathcal L$ is
\begin{equation}
\sigma_{\rm pt}(\mathcal L) \,\,:=\,\, \{ \lambda \in \mathbb C \, |\, 
\operatorname{Ker}(\mathcal L - \lambda) \neq \{0\} \}.
\end{equation}
The \emph{Fredholm point spectrum} of $\mathcal L$ is the subset of 
$\sigma_{\rm pt}(\mathcal L)$ defined by
\begin{equation}
\sigma^{\mathcal F}_{\rm pt}(\mathcal L) \,\,:=\,\, \{ \lambda \in \mathbb C \, |\, 
\mathcal L - \lambda \text{ is Fredholm, }
\operatorname{Ind}(\mathcal L - \lambda)=0, 
\text{ and } \operatorname{Ker}(\mathcal L - \lambda) \neq \{0\} \},
\end{equation}
and the \emph{essential spectrum} of $\mathcal L$ is 
$\sigma_{\rm{ess}}(\mathcal L) \,\,:=\,\, \sigma(\mathcal L) \setminus 
\sigma^{\mathcal F}_{\rm{pt}}(\mathcal L)$. 
We observe that $\sigma(\mathcal L) = \sigma_{\rm pt}(\mathcal L) \cup 
\sigma_{\rm{ess}}(\mathcal L)$, but note that this union may not be disjoint. 
\end{definition} 

\begin{remark}
Since the operators we consider are not self-adjoint, there are several
non-equivalent definitions of the essential spectrum~\cite{edmunds2018spectral}.
An argument that involves the 
closed graph theorem~\cite{kato2013perturbation} shows that 
with the definition we use, 
$\sigma_{\rm{ess}}(\mathcal L)$
consists of those $\lambda\in\mathbb C$ so that either
(a) $\mathcal L - \lambda$ is Fredholm but has 
$\operatorname{Ind}(\mathcal L - \lambda) \neq 0$ or (b) $\mathcal L - \lambda$ is  not Fredholm.
This definition  gives the largest subset of the spectrum that is invariant under compact perturbations~\cite{edmunds2018spectral}. However, the operators, $\mathcal L$, we 
consider are given as perturbations of constant coefficient differential operators, 
$\mathcal L_\infty$, by multiplication operators, which are not compact.
Nevertheless, we will show below that the operator
 $\mathcal L(t)$ is a \emph{relatively compact perturbation} of $\mathcal L_\infty$,
 by which we mean that $\exists \lambda \in \rho(\mathcal L_\infty)$ so that
$(\mathcal L(t) - \mathcal L_\infty) (\mathcal L_\infty-\lambda)^{-1} : X \to X$ is compact.
Consequently, by Weyl's essential spectrum theorem~\cite{kapitula2013spectral},
$\sigma_{\rm{ess}}(\mathcal L(t)) = \sigma_{\rm{ess}}(\mathcal L_\infty)$.
\end{remark}

To calculate the spectrum of the asymptotic operator, $\mathcal L_\infty$,
we convert the equation, $(\mathcal L_\infty - \lambda)\mathbf p = \mathbf 0$,
 to the first-order system,
 \begin{equation}
 \partial_x \mathbf Y \,\,=\,\, \mathbf A_\infty(\lambda) \mathbf Y,
 \label{eq:Y}
 \end{equation}
 where 
 $\mathbf Y = \begin{bmatrix} \mathbf p \\ \mathbf p_x\end{bmatrix} \in \mathbb C^4$
 and $\mathbf A_\infty(\lambda)\in \mathbb C^{4\times 4}$ is the constant 
 matrix
 \begin{equation}
 \mathbf A_\infty(\lambda) \,\, := \,\, 
 \begin{bmatrix} 
 \mathbf{0} & \mathbf{I} \\ 
 \mathbf B^{-1}(\lambda - \mathbf M_\infty) & \mathbf{0}
 \end{bmatrix}.
 \end{equation}
  
\begin{proposition}\label{prop:EssSpecLinfty}
Assume that Hypothesis~\ref{hyp:EA} is met. 
Then, $\sigma_{\rm{pt}}(\mathcal L_\infty) = \emptyset$ and 
\begin{equation}
\sigma_{\rm{ess}}(\mathcal L_\infty) \,\,=\,\, \{ \lambda \in \mathbb C\, |\, \exists \mu\in \mathbb R: \operatorname{det}[\lambda - \mathbf{M}_\infty + \mu^2 \mathbf{B}]=0 \}.
\end{equation}
\end{proposition}

\begin{proof}
Since \eqref{eq:Y} does not have solutions that decay as $x\to +\infty$ and as 
 $x\to -\infty$, $\sigma_{\rm{pt}}(\mathcal L_\infty) = \emptyset$.
Arguing as in the proof of Kapitula and Promislow~\cite[Lemma 3.1.10]{kapitula2013spectral}, we have that 
$\lambda \in \sigma_{\rm{ess}}(\mathcal L_\infty)$ precisely when 
the matrix $\mathbf A_\infty(\lambda)$ has a pure imaginary eigenvalue, that is
\begin{equation}\label{eq:EssSpecLinfty}
\sigma_{\rm{ess}}(\mathcal L_\infty) \,\,=\,\,
\{ \lambda \in \mathbb C\, |\, \exists \mu\in \mathbb R: \operatorname{det}
[\mathbf{A}_\infty(\lambda) - i\mu]=0 \}.
\end{equation}
Premultiplying $\mathbf{A}_\infty(\lambda) - i\mu$ by the invertible matrix
$\begin{bmatrix} \mathbf 0 & \mathbf B \\ \mathbf B & \mathbf 0 \end{bmatrix}$,
and applying the Schur determinental formula~\cite{meyer2000matrix},
we find that $\lambda \in \sigma_{\rm{ess}}(\mathcal L)$ if and only if
\begin{equation}
0 \,\,=\,\, 
\operatorname{det} \begin{bmatrix} \lambda - \mathbf{M}_\infty & -i\mu \mathbf B \\
-i\mu\mathbf{B} & \mathbf{B} \end{bmatrix}
\,\,=\,\,\operatorname{det}\mathbf{B}
\operatorname{det}[\lambda - \mathbf{M}_\infty + \mu^2 \mathbf{B}].
\end{equation}
\qed
\end{proof}

\begin{example}[KM Breather]
For the KM breather discussed in Example~\ref{ex:KM},
 \begin{equation}
 \mathbf{M}_\infty = \begin{bmatrix}0&0\\2\nu_0^2&0\end{bmatrix}.
 \end{equation}
Consequently, 
\begin{equation}
\sigma_{\rm{ess}}(\mathcal L_\infty)
\,\,=\,\,
\{ \lambda \in \mathbb C \,\,|\,\, \lambda = \pm i \tfrac \mu 2 \sqrt{\mu^2-4\nu_0^2}
\text{ for some } \mu\in\mathbb R \},
\end{equation}
which is the union of the imaginary axis and the interval $[-\nu_0,\nu_0]$ in the 
real axis.
This result is consistent with a result of  Cuevas-Maraver \emph{et al.}~\cite{cuevas2017floquet}, who performed a modulation
instability analysis to show that the frequency, $\omega$, and wave number, $k$,
of a perturbation about a plane wave background, $\psi_\infty(t,x)=\nu_0$,
satisfy the dispersion relation 
$\omega = \pm \frac k 2 \sqrt{k^2-4\nu_0^2}$.
\end{example}

\begin{example}[CQ-CGL Breathers]
For the CQ-CGL breathers discussed in Example~\ref{ex:GL}, 
 $\mathbf{M}_\infty =  \delta \mathbf{I}$. 
 Consequently, 
\begin{equation}
\sigma_{\rm{ess}}(\mathcal L_\infty)
\,\,=\,\,
\{ \lambda \in \mathbb C \,\,|\,\, \lambda = \delta - \mu^2 (\beta \pm i D/2) 
\text{ for some } \mu\in\mathbb R \},
\end{equation}
which is a pair of half-lines in the complex plane that are symmetric 
about the real axis~\cite{PhysicaD116p95,shen2016spectra}. 
If the physical system being modeling includes
linear loss ($\delta < 0$) and a spectral filter ($\beta>0$), then the essential spectrum
is stable. 
\end{example}

We conclude this section by showing that
$\mathcal L(t)$ is a relatively compact perturbation of $\mathcal L_\infty$.
This result plays an important role in the proof of our main result, 
Theorem~\ref{thm:EssSpecU}, on the essential spectrum of the monodromy operator. 

\begin{theorem}\label{thm:EssSpec}
Assume that Hypothesis~\ref{hyp:EA} is met. 
Then,  the differential operator, 
$\mathcal L(t)$, given in \eqref{eq:OpL}, is a relatively compact perturbation of 
 $\mathcal L_\infty$.
\end{theorem}

Versions of this result are well known folklore. For example, Kapitula and Promislow~\cite[Theorem 3.1.11]{kapitula2013spectral} provide a proof 
in the special case that  $\boldsymbol\psi - \boldsymbol\psi_\infty$ is 
compactly supported.  
For the sake of completeness, and because it may be of independent
interest to some readers, we provide a proof of the general case. This proof is based on a 
characterization of compact subsets of $L^2(\mathbb R)$ due to Kolmogorov and Riesz~\cite{hanche2010kolmogorov,reed1980methods}.

In the sequel, for each Banach space, $X$, we let $\| \cdot \|_X$ denote the standard
norm on $X$.

\begin{proof}
Let $\lambda\in \rho(\mathcal L_\infty)$. We must show that the operator
$\mathcal K :=   \mathbf M(t) \circ (\mathcal L_\infty - \lambda)^{-1}
\,\,:\,\, L^2(\mathbb R,\mathbb C^2) \to H^2(\mathbb R,\mathbb C^2)
\to L^2(\mathbb R,\mathbb C^2)$ is compact. It suffices to show that
for any bounded family of functions, $\mathcal H\subset L^2(\mathbb R,\mathbb C^2) $,
the subset $\mathcal F = \mathcal K(\mathcal H)\subset L^2(\mathbb R,\mathbb C^2) $
is precompact, or equivalently is  totally bounded. A classical theorem due to Kolmogorov and Riesz~\cite{hanche2010kolmogorov} states that
$\mathcal F \subset L^2(\mathbb R,\mathbb C^2) $ is totally bounded  if and only if
the following three conditions hold:

\begin{enumerate}
\item $\mathcal F $ is bounded,
\item for all $\epsilon > 0$ there is an $R>0$ so that for all $f \in \mathcal F$,
\begin{equation}
\int\limits_{|x|>R} \| f(x) \|^2_{\mathbb C^2}\, dx \,\, < \,\, \epsilon^2,\qquad\text{and}
\end{equation}
\item for all $\epsilon >0$ there is a  $\delta>0$ so that for all $f \in \mathcal F$ and
$y\in\mathbb R$ with $|y|<\delta$, 
\begin{equation}
\int\limits_{\mathbb R} \| f(x+y) -f(x)\|^2_{\mathbb C^2}\, dx \,\, < \,\, \epsilon^2.
\end{equation}
\end{enumerate}

The first condition holds as the subset $\mathcal H$
and the operator $\mathcal K$ are bounded. For the second condition, we first
observe that there is a $C>0$ so that for all $h\in\mathcal H$,
\begin{equation}
\| (\mathcal L_\infty - \lambda)^{-1} h \|_{H^2(\mathbb R,\mathbb C^2)} < C.
\end{equation}
Next, by Hypothesis~\ref{hyp:EA},  there is an $\widetilde R >0$ so that
$\| \mathbf M(t,x)\|_{\mathbb C^{2\times 2}} < e^{-r|x|}/C$ for all $|x| > \widetilde R$.
Let $R>\widetilde R$. Since every $f\in \mathcal F = \mathcal K(\mathcal H)$
is of the form  $f=\mathbf M(t)g$ for some 
$g\in (\mathcal L_\infty - \lambda)^{-1}(\mathcal H)$,
\begin{equation}
\int_{|x|>R} \| f(x) \|^2_{\mathbb C^2}\, dx \,\,<\,\, \frac{1}{C^2} e^{-2rR} \|g\|_{H^2(\mathbb R,\mathbb C^2)} \,\,<\,\, \epsilon^2,
\end{equation}
provided  that $R > |\log\epsilon|/r$.\footnote{
We observe that the third condition 
in Hypothesis~\ref{hyp:EA} that
$\boldsymbol\psi$ decays exponentially can be significantly relaxed. In particular, 
we do not even require that $\boldsymbol\psi$ or $\boldsymbol\psi_x$ 
belong to $L^2(\mathbb R,\mathbb C^2)$.
}

For the third condition, we first observe that   $f=\mathbf M(t)g\in H^1(\mathbb R,\mathbb C^2)$,  by Hypothesis~\ref{hyp:EA} and  since $g\in H^2(\mathbb R,\mathbb C^2)$. 
Appealing to a result
in Evans~\cite[\S 5.8.2]{evans2010partial} on difference quotients for $H^1$ functions, we
have that
\begin{align*}
&\,\, \,\,\,\,\,\,\int_{\mathbb R} \| f(x+y) - f(x) \|^2_{\mathbb C^2} \, dx 
 \leq |y|^2 \, \| f' \|^2_{L^2(\mathbb R,\mathbb C^2)}  \\
& \leq |y|^2 \operatorname{max} \{ \| \mathbf M(t)\|^2_{L^\infty(\mathbb R, \mathbb C^2)},
\| \mathbf M_x(t)\|^2_{L^\infty(\mathbb R, \mathbb C^2)} \} \, \| g \|^2_{H^2(\mathbb R, \mathbb C^2)},
\end{align*}
which can be made arbitrarily small, provided $y$ is close enough to zero.\qed
\end{proof}

\section{The Evolution Family}\label{Sec:EvolFamily}

In this section, we establish the existence of an evolution family for  the linearized 
equation~\eqref{eq:Linearization} by applying classical results on the existence, uniqueness, and regularity of 
solutions of initial-value problems for non-autonomous linear differential equations in Banach spaces. 
In Section~\ref{Sec:EssSpecM}, the evolution family will be used to define the monodromy operator.

We study solutions, $\mathbf u:[s,\infty) \to H^2(\mathbb R,\mathbb C^2)$, 
of the  initial-value problem 
\begin{equation}
\begin{aligned}
\frac{\partial \mathbf u}{\partial t} \,\,&=\,\, \mathcal L(t) \mathbf u, \qquad \text{for } t>s,\\
\mathbf u(s) &= \mathbf v,
\end{aligned}
\label{eq:IVP}
\end{equation}
where $\mathcal L(t) = \mathbf B \,\partial_x^2 + \widetilde{\mathbf M}(t)$ is the
$t$-dependent family of operators on $L^2(\mathbb R, \mathbb C^2)$ defined in~\eqref{eq:OpL}, and
$\mathbf v \in H^2(\mathbb R,\mathbb C^2)$.

The existence of solutions of  linear differential equations
in Banach spaces is typically established using 
semigroup theory.
Since the operator, $\mathcal L(t)$, is $t$-dependent, we utilize the theory of
evolution families~\cite{pazy2012semigroups}, 
in which the solution of~\eqref{eq:IVP}
is represented in the form, $\mathbf u(t) = \mathcal U(t,s)\mathbf v$, where $\mathcal U(t,s)$ is an evolution operator. The following theorem establishes conditions 
on a  solution, $\psi=\psi(t,x)$, of the nonlinear wave equation~\eqref{eq:NWE}
that ensure the existence of an evolution operator, $\mathcal U(t,s)$, for the
linearized equation, \eqref{eq:IVP}. 

\begin{hypothesis}\label{hyp:PsiEvOp}
The $T$-periodic solution, $\boldsymbol\psi$, about which \eqref{eq:NWEsys} is 
linearized has the property that
both $\boldsymbol\psi$ and $\boldsymbol\psi_t$ are 
bounded and are  continuous on $[0,T]\times \mathbb R$,
uniformly in $x$. 
\end{hypothesis}

\begin{definition}
Let   
$\mathbf A =\mathbf A(t,x): [0,\infty)\times \mathbb R \to \mathbb C^{2\times 2}$
be a bounded matrix-valued function. We define
 \begin{equation}
 \| \mathbf A \|_\infty \,\,:=\,\, \sup\limits_{(t,x)} \| \mathbf A (t,x)\|_{\mathbb C^{2\times 2}}.
 \end{equation} 
\end{definition}

\begin{theorem}\label{thm:EvOp}
Assume that Hypothesis~\ref{hyp:PsiEvOp} is met 
and that  $\beta \geq 0$. 
Then, there exists a unique evolution operator, $\mathcal U(t,s) \in \mathcal B(L^2(\mathbb R, \mathbb C^2))$, 
for $0\leq s \leq t < \infty$, such that
\begin{enumerate}
\item $\| \mathcal U(t,s) \|_{\mathcal B(L^2(\mathbb R, \mathbb C^2))} \leq 
\exp [ \, \| \widetilde{\mathbf M} \|_\infty (t-s) \,] $,
\item $\mathcal U(t,s) (H^2(\mathbb R, \mathbb C^2)) \subset 
H^2(\mathbb R, \mathbb C^2)$, and
\item For each $s$, $\mathcal U(\cdot,s)$ is strongly continuous in that
for all $\mathbf v\in L^2(\mathbb R, \mathbb C^2)$, the mapping
$t\mapsto \mathcal U(t,s)\mathbf v$ is continuous, and
\item For each $\mathbf v \in H^2(\mathbb R, \mathbb C^2)$, the function 
$\mathbf u(t) = \mathcal U(t,s)\mathbf v$ is the unique solution of the initial value problem~\eqref{eq:IVP} for which $\mathbf u \in C([s,\infty),H^2(\mathbb R, \mathbb C^2))$ 
and $\mathbf u \in C^1((s,\infty),L^2(\mathbb R, \mathbb C^2))$.
\end{enumerate}
\end{theorem}

\begin{remark}
In the case that $\beta > 0$, for any $\mathbf v \in L^2(\mathbb R, \mathbb C^2)$
we have that
$\mathbf u(t) = \mathcal U(t,s)\mathbf v \in  H^2(\mathbb R, \mathbb C^2)$.
However, when  $\beta = 0$, we require that $\mathbf v \in H^2(\mathbb R, \mathbb C^2)$
in order that $\mathbf u(t) \in H^2(\mathbb R, \mathbb C^2)$.
\end{remark}

\begin{proof}
The theorem follows from a combination of the 
Hille-Yosida Theorem (see Theorem~3.1 of Pazy~\cite[Ch. 1]{pazy2012semigroups}) and Theorems~2.3 and 4.8 of \cite[Ch. 5]{pazy2012semigroups}. The following three
lemmas ensure that
the hypotheses of these three theorems hold. 

\begin{lemma}\label{lemma:HY}
The linear operator, $\mathbf B \partial_{x}^2: H^2(\mathbb R,\mathbb C^2) \subset 
L^2(\mathbb R,\mathbb C^2) \to L^2(\mathbb R,\mathbb C^2)$  is closed with domain $H^2(\mathbb R,\mathbb C^2)$. Furthermore, $\mathbb R^+ \subset \rho( \mathbf B \partial_{x}^2 )$ and
the resolvent operator (see Definition~\ref{SpecDefs}) satisfies
\begin{equation}\label{eq:HYResBound}
\| R(\lambda : \mathbf B \partial_{x}^2)\|_{\mathcal B(L^2(\mathbb R,\mathbb C^2))} 
\,\,\leq \,\,
\frac 1\lambda \qquad\text{for all } \lambda > 0. 
\end{equation}
Consequently, by the Hille-Yosida Theorem, $\mathbf B \partial_{x}^2$ is the infinitesimal generator of a $C_{0}$-semigroup of contractions on $L^2(\mathbb R,\mathbb C^2)$.
\end{lemma}

\begin{proof}
The closedness of the operator $\mathbf B\partial_x^2$ 
is discussed in Remark~\ref{remark:spaces}. 
By Theorem~\ref{thm:EssSpec}, $\sigma(\mathbf B\partial_x^2) 
= \{ s(\beta \pm iD/2) \, | \, s\leq 0\}$, and so $\mathbb R^+ \subset 
\rho(\mathbf B\partial_x^2)$. 
To establish the bound~\eqref{eq:HYResBound} on the norm of the resolvent, we
use the Fourier transform, 
$\mathcal F: L^2(\mathbb R,\mathbb C^2) \to L^2(\mathbb R,\mathbb C^2)$, defined by
\begin{equation}
\widehat f(\xi)\,\,=\,\,{\mathcal F}[f](\xi)\,\,=\,\, \int_{\mathbb R} f(x) \exp(-2\pi i x\xi)\,dx.
\end{equation}
Now, for $\mathbf v \in L^2(\mathbb R,\mathbb C^2)$
\begin{equation}
\mathcal F[R(\lambda: \mathbf B \partial_x^2)\mathbf v] (\xi)
\,\,=\,\, \mathbf C(\xi) \widehat{\mathbf v}(\xi),
\end{equation}
where
\begin{equation}
\mathbf C(\xi)  \,\,=\,\, ( -4\pi\xi^2 \mathbf B - \lambda)^{-1}
\,\,=\,\, \frac {1}{d(\xi)} 
\begin{bmatrix}
-4\pi\xi^2\beta - \lambda & -4\pi \xi^2 D/2 \\
4\pi \xi^2 D/2 & -4\pi\xi^2\beta - \lambda
\end{bmatrix},
\end{equation}
with $d(\xi) = (4\pi \xi^2\beta + \lambda)^2 + (4\pi\xi^2 D/2)^2$.
By Parseval's Theorem,
\begin{align*}
\| R(\lambda:\mathbf B \partial_x^2)\mathbf v\|_{L^2(\mathbb R,\mathbb C^2)}^2
\,\,&=\,\, \| \mathbf C\widehat{\mathbf v}\|_{L^2(\mathbb R,\mathbb C^2)}^2
\\
\,\,&\leq\,\, \int_{\mathbb R} \| \mathbf C(\xi)\|_{\mathbb C^{2 \times 2}}^2
\| \widehat{\mathbf v}(\xi)\|_{\mathbb C^2}^2\, d\xi
\,\,\leq\,\,
\| \mathbf C \|_\infty^2  \|    {\mathbf v}\|_{L^2(\mathbb R,\mathbb C^2)}^2,
\end{align*}
where $\| \mathbf C \|_\infty := \sup\limits_{\xi\in\mathbb R} \| \mathbf C(\xi)\|_{\mathbb C^{2\times 2}}$.
Since $\| \mathbf C(\xi)\|_{\mathbb C^{2\times 2}}^2$ is equal to the largest
eigenvalue of $\mathbf C(\xi)^T \mathbf C(\xi) = d(\xi)^{-1}\mathbf I_{2\times 2}$,
we conclude that 
$\| \mathbf C(\xi)\|_{\mathbb C^{2\times 2}} = d(\xi)^{-1/2} \leq \lambda^{-1}$,
for all $\lambda > 0$.
Since $\beta>0$, we conclude that 
$\| R(\lambda: \mathbf B \partial_{x}^2)\|_{\mathcal B(L^2(\mathbb R,\mathbb C^2))} 
\,\,\leq \,\,\lambda^{-1}$, as required. 
\qed
\end{proof}

\begin{lemma}\label{lemma:Mbound}
Assume that Hypothesis~\ref{hyp:PsiEvOp} is met. Then, for all $t>0$,
\begin{equation}
\| \widetilde{\mathbf M}(t) \|_{\mathcal B(L^2(\mathbb R,\mathbb C^2))} 
\,\,\leq \,\, \| \widetilde{\mathbf M} \|_\infty \,\, < \,\, \infty.
\end{equation}
\end{lemma}

\begin{proof}
Let $\mathbf u\in L^2(\mathbb R,\mathbb C^2)$. Then,
\begin{align*}
\| \widetilde{\mathbf M}(t) \mathbf u \|_{L^2(\mathbb R,\mathbb C^2)}^2
\,\,\leq \,\, \int_{\mathbb R} \| \widetilde{\mathbf M}(t,x)\|_{\mathbb C^{2\times 2}}^2
\|\mathbf u(x)\|^2_{\mathbb C^2}\, dx
\,\,\leq\,\, \| \widetilde{\mathbf M}\|^2_\infty\, \| \mathbf u\|^2_{L^2(\mathbb R,\mathbb C^2)}.
\end{align*}
To show that $\| \widetilde{\mathbf M}\|_\infty < \infty$ we use 
Hypothesis~\ref{hyp:PsiEvOp} together with the fact that 
the matrix 2-norm is bounded by the Frobenius norm, 
$\| \mathbf A \|_{\mathbb C^{2\times 2}} \leq \| \mathbf A \|_F$, where 
$\| \mathbf A \|_F^2:=\sum\limits_{i,j=1}^2 |A_{ij}|^2$. 
\qed
\end{proof}

Taken together, Theorem 2.3 of \cite[Ch. 5]{pazy2012semigroups}  and
 Lemmas~\ref{lemma:HY} and \ref{lemma:Mbound} imply that 
 $\{\mathcal L(t)\}_{t\geq 0}$ is a stable family of infinitesimal generators
 of   $C_0$-semigroups on $L^2(\mathbb R,\mathbb C^2)$,
 which is one of the assumptions required
 for the application of Theorem~4.8 of Pazy~\cite[Ch. 5]{pazy2012semigroups}.
  The remaining assumption in that theorem 
 also holds thanks to the following lemma. 

\begin{lemma}\label{lemma:LisC1}
Assume that Hypothesis~\ref{hyp:PsiEvOp} is met. Then,
for each $\mathbf v\in H^2(\mathbb R,\mathbb C^2)$, we have that
$F(\cdot) := \mathcal L(\cdot)\mathbf v: (0,\infty)\to L^2(\mathbb R,\mathbb C^2)$
is $C^1$. 
\end{lemma}

\begin{proof}
We show that $F$ is differentiable with $F'(t) = \partial_t\widetilde{\mathbf M}(t)\mathbf v$.
The proof that $F'$ is continuous is similar. 
First we observe that since $\boldsymbol\psi$ and $\boldsymbol\psi_t$ are assumed to be
bounded, $ \mathcal L(t)\mathbf v$ and $\partial_t\widetilde{\mathbf M}(t)\mathbf v$
are in $L^2(\mathbb R,\mathbb C^2)$. Next, we apply the fundamental theorem
of calculus and use the equivalences between the matrix 2-norm and the Frobenius norm
to obtain the estimate
\begin{align*}
&\| F(t+h) - F(t) - h F'(t) \|_{L^2(\mathbb R,\mathbb C^2)}^2 \\
&\leq\,\,
\int_{\mathbb R} \| \widetilde{\mathbf M}(t+h,x) - \widetilde{\mathbf M}(t,x) 
- h \partial_t\widetilde{\mathbf M}(t,x) \|_{\mathbb C^{2\times 2}}^2 
\|\mathbf v(x)\|_{\mathbb C^2}^2\, dx
\\
&=\,\, \int_{\mathbb R} \left\| \int_t^{t+h} \left[ 
\partial_t \widetilde{\mathbf M}(\tau,x) - \partial_t \widetilde{\mathbf M}(t,x)\right]
\, d\tau \right\|_{\mathbb C^{2\times 2}}^2
\|\mathbf v(x)\|_{\mathbb C^2}^2\, dx
\\
&\leq\,\, h^2 \sum_{i,j=1}^2 \sup\limits_{(\tau,x)\in (t,t+h)\times \mathbb R}  
\left( \left| \partial_t \widetilde{\mathbf M}_{i,j}(\tau,x) - 
\partial_t \widetilde{\mathbf M}_{i,j}(t,x) 
\right|^2\right) 
\, \| \mathbf v\|_{L^2(\mathbb R,\mathbb C^2)}
\\
&\leq\,\, 8h^2  \sup\limits_{(\tau,x)\in (t,t+h)\times \mathbb R}  
\left\| \partial_t \widetilde{\mathbf M}(\tau,x) - 
\partial_t \widetilde{\mathbf M}(t,x) 
\right\|_{\mathbb C^{2\times 2}}^2
\, \| \mathbf v\|_{L^2(\mathbb R,\mathbb C^2)}.
\end{align*}
The result now follows, since by Hypothesis~\ref{hyp:PsiEvOp},
$\partial_t \widetilde{\mathbf M}$ is continuous in $t$, uniformly in $x$. 
\qed
\end{proof}   
\qed
\end{proof}

\section{The essential spectrum of the monodromy operator [Case: $\beta>0$]}\label{Sec:EssSpecM}

In this section, we  define the monodromy operator that is the main focus
of this paper, and derive a formula for its essential spectrum  
in terms of that of the asymptotic operator, $\mathcal L_\infty$. 
Our proof relies on the fact 
that the semigroup, $e^{t \mathcal L_\infty}$, associated with 
$\mathcal L_\infty$ is analytic, which only holds when $\beta>0$. Therefore,
although the results
in this section apply to the CQ-CGL breathers
 in  Example~\ref{ex:GL}, 
they do not apply to the KM breather in Example~\ref{ex:KM}. 

\begin{definition}
Let $\psi$ be a periodically-stationary
solution of the constant coefficient CQ-CGL equation~\eqref{eq:NWE}
and let $s\in\mathbb R$.
The \emph{monodromy operator}, $\mathcal M(s)$, associated with 
the linearization \eqref{eq:Linearization} of \eqref{eq:NWE} about 
$\psi$ is the bounded operator on $L^2(\mathbb R,\mathbb C^2)$
defined by $\mathcal M(s)=\mathcal U(s+T,s)$, where 
$\mathcal U$ is the evolution operator of Theorem~\ref{thm:EvOp}
and $T$ is the period of $\psi$. 
\end{definition}

\begin{theorem}\label{thm:EssSpecU}
Suppose that $\beta > 0$ and that
Hypotheses~\ref{hyp:EA} and \ref{hyp:PsiEvOp} are met. 
Then the $C^0$-semigroup, $e^{t\mathcal L_\infty}$,  generated by
$\mathcal L_\infty$ is analytic. 
Furthermore,
the essential spectrum of the evolution operator, $\mathcal U(t,s)$, is given by
\begin{equation}
\sigma_{\rm{ess}}(\mathcal U(t,s)) \,\,=\,\, \sigma_{\rm{ess}}(e^{(t-s)\mathcal L_\infty}).
\end{equation}
Therefore, the essential spectrum of the monodromy operator, $\mathcal M(s)$, is given by
\begin{equation}
\sigma_{\rm{ess}}(\mathcal M(s)) \,\,=\,\, \sigma_{\rm{ess}}(e^{T\mathcal L_\infty}),
\end{equation}
which is independent of $s$. 
\end{theorem}

Before proving this theorem, we state and prove a Corollary.

\begin{corollary}\label{cor:EssSpecM}
Suppose that $\beta > 0$ and Hypothesis~\ref{hyp:PsiEvOp} is met.
Then, the essential spectrum of the monodromy operator, $\mathcal M(s)$, is given by
\begin{equation}
\sigma_{\rm{ess}}(\mathcal M(s)) \setminus \{0\} \,\,=\,\, e^{T\sigma_{\rm{ess}}(\mathcal L_\infty)}.
\end{equation}
In particular, in the case of a CQ-CGL breather, if $\delta < 0$ then
the  essential spectrum lies inside a circle of radius $e^{\delta T} < 1$.
\end{corollary}

\begin{proof}

Since the spectral mapping theorem holds for the point spectrum 
of a $C_0$-semigroup~\cite{engelnagel2000},
$\sigma_{\rm{pt}}(e^{T\mathcal L_\infty}) \setminus \{ 0\} = e^{T \sigma_{\rm{pt}}
(\mathcal L_\infty)} = \emptyset $, by Theorem~\ref{thm:EssSpec}. 
Since $\sigma^{\mathcal F}_{\rm{pt}} \subset 
\sigma_{\rm{pt}}$, 
$\sigma^{\mathcal F}_{\rm{pt}}(e^{T\mathcal L_\infty}) \setminus \{0\}= \emptyset$. Consequently, 
$\sigma_{\rm{ess}}(e^{T\mathcal L_\infty}) \setminus \{ 0\} = 
\sigma(e^{T\mathcal L_\infty}) \setminus \{ 0\} = 
e^{T \sigma (\mathcal L_\infty)}$, since the 
spectral mapping theorem also holds for the spectrum of an
analytic semigroup~\cite{engelnagel2000}. The result now 
follows since  $\sigma (\mathcal L_\infty) =  \sigma_{\rm{ess}} (\mathcal L_\infty)$.
\qed
\end{proof}

The proof of Theorem~\ref{thm:EssSpecU} relies on the following two lemmas.

\begin{lemma}
Suppose that Hypothesis~\ref{hyp:PsiEvOp} is met. Then
\begin{equation}
\mathcal U(t,s) \,\,=\,\, e^{(t-s)\mathcal L_\infty} \,\,+\,\, \int_s^t 
e^{(t-\tau)\mathcal L_\infty} \,\mathbf M(\tau) \, \mathcal U(\tau,s)\, d\tau,
\qquad\text{in } \mathcal B(L^2(\mathbb R,\mathbb C^2)).
\label{eq:VarOfParU}
\end{equation}
\end{lemma}

\begin{proof}
We refer to Engel and Nagel~\cite[App. C]{engelnagel2000} for a summary of the
theory of Lebesgue integration for
functions, $f:J\to X$, from an interval $J\subset \mathbb R$ to a Banach space, $X$. 
As in the proof of Theorem~\ref{thm:EvOp}, the
asymptotic operator, $\mathcal L_\infty$, generates a $C_0$-semigroup 
on $L^2(\mathbb R, \mathbb C^2)$. For each $\tau\in[s,t]$
and  $\mathbf v\in H^2(\mathbb R,\mathbb C^2)$, 
let 
$f(\tau) :=  \mathbf M(\tau) \,\mathcal U(\tau,s)\mathbf v \in  L^2(\mathbb R,\mathbb C^2)$. By Hypothesis~\ref{hyp:PsiEvOp}, $f\in L^1([s,t], L^2(\mathbb R,\mathbb C^2))$. 
Since $\mathcal L(t) = \mathcal L_\infty + \mathbf M(t)$, the result follows from
the variation of parameters formula (see
Corollary 2.2 of \cite[Ch. 4]{pazy2012semigroups}, together with Theorem~\ref{thm:EvOp} above).
\end{proof}

 Corollary 10.6 of~\cite[Ch. 1]{pazy2012semigroups} implies that 
 $(e^{t\mathcal L_\infty})^* = e^{t\mathcal L_\infty^*}$, where
 $\mathcal L_\infty^* = \mathbf B^T\,\partial^2_x + \mathbf M_\infty^T$ is the adjoint of $\mathcal L_\infty$.

\begin{lemma}\label{lemma:AnalSemiGp}
Suppose that $\beta > 0$ and Hypothesis~\ref{hyp:PsiEvOp} is met.
Then the semigroups $e^{t\mathcal L_\infty}$  and $e^{t\mathcal L_\infty^*}$ are analytic. 
\end{lemma}

\begin{proof}
We will show that for all $\sigma >0$ and $\tau\neq 0$, 
\begin{equation}
\| R(\sigma + i \tau : \mathbf B \partial_x^2) \|\,\, \leq \,\,\frac{\sqrt{1+(D/2\beta)^2}}{|\tau|}.
\label{eq:ASGResolventCond}
\end{equation}
Therefore, by Theorem~5.2 of \cite[Ch. 2]{pazy2012semigroups} (and the discussion
preceeding it), $\mathbf B \partial_x^2$
is the infinitesimal generator of an analytical semigroup. Since 
$\mathcal L_\infty =  \mathbf B \partial_x^2 + \mathbf M_\infty$, where $ \mathbf M_\infty$
 is a bounded operator, it
follows from Corollary~2.2 of \cite[Ch. 3]{pazy2012semigroups} that 
$\mathcal L_\infty$ is the infinitesimal generator of an analytical semigroup. 
The same argument holds for the adjoint. Note that as $\beta \to 0$, the
constant on the right-hand side of \eqref{eq:ASGResolventCond} 
blows up. Consequently, this proof cannot be extended to the case $\beta=0$.

To establish \eqref{eq:ASGResolventCond}, 
as in the proof of Lemma~\ref{lemma:HY}, 
and with $\lambda = \sigma + i \tau$,
we observe that
\begin{align}
\| R(\sigma + i \tau : \mathbf B \partial_x^2) \|^2
 \,\,&\leq\,\, \sup_{\xi\in \mathbb R}\mu_{\text{max}} [\mathbf C(\xi)^* \mathbf C(\xi)] \\
 \,\,&=\,\, \left( \inf_{\xi\in \mathbb R} 
  \mu_{\text{min}} [(-4\pi\xi^2\mathbf B - \lambda)
( -4\pi\xi^2\mathbf B^T - \overline{\lambda})]\right)^{-1},
\end{align}
since  the largest eigenvalue of a non-negative definite Hermitian matrix
is the inverse of the smallest eigenvalue of its inverse.
Let $a= 4\pi\xi^2\beta>0$ and $b= 2\pi\xi^2 D$. 
A calculation shows that 
$\mu_\text{min} = |a+\lambda|^2 + b^2 - 2|b\tau|$.
If we let $z=  -a +i\operatorname{sgn}(\tau b) b$, then
$\mu_\text{min} = |z-\lambda|^2$. 
The points, $z$, lie on the half-line in the left-half plane given by
$y= - \operatorname{sgn}(\tau) mx$ with $x<0$, where $m= | D/2\beta |$, 
while $\lambda$ is in the right half-plane. We observe that $z$ and $\lambda$
either  both lie above or both lie below the real axis. 
 Consequently, the problem 
of minimizing $\mu_\text{min}$  as a  function of $\xi$
is that of finding the minimum of the
square of the distance from 
$\lambda$  to this half-line, which is greater than the minimum squared 
distance to the full line. Therefore,
\begin{equation}
\| R(\sigma + i \tau : \mathbf B \partial_x^2) \| \,\,\leq \,\,\frac{\sqrt{1+m^2}}{m\sigma + |\tau|}
\,\,\leq\,\, \frac{\sqrt{1+m^2}}{|\tau|},
\end{equation}
as required.
\qed
\end{proof}

\noindent\emph{Proof of Theorem~\ref{thm:EssSpecU}}
Since the essential spectrum is invariant under compact perturbations, to prove the theorem, we just need to show that the integral in \eqref{eq:VarOfParU} is a compact
operator. 
By Engel and Nagel~\cite[Theorem C.7]{engelnagel2000}, it suffices to show that
for each 
$\tau\in(s,t)$ 
the integrand
$\mathcal K(\tau) =  e^{(t-\tau)\mathcal L_\infty} \,\mathbf M(\tau) \, \mathcal U(\tau,s)$
 is compact, and that
the function 
$\mathcal K:(s,t)\to \mathcal B(L^2(\mathbb R,\mathbb C^2))$ 
is strongly
continuous, in that for all $\mathbf v \in L^2(\mathbb R,\mathbb C^2)$,
$\|\mathcal K(\tau)\mathbf v-\mathcal K(\tau_0)\mathbf v\|_{L^2(\mathbb R,\mathbb C^2)}
\to 0$ as $\tau\to\tau_0$.

To show that $\mathcal K(\tau)$ is compact we show that
the adjoint, $\mathcal K^*(\tau)$, is compact. 
As in Theorem~\ref{thm:EssSpec}, 
$\mathcal L^*(\tau)$ is a relatively compact perturbation of  $\mathcal L_\infty^*$.
Therefore there is a $\lambda \in \rho(\mathcal L_\infty^*)$ so that
$\mathbf M^*(\tau) (\mathcal L_\infty^* - \lambda)^{-1} \in 
\mathcal B(L^2(\mathbb R,\mathbb C^2))$ is compact.
Since the semigroup $e^{t\mathcal L_\infty^*}$ is analytic, 
the operator $(\mathcal L_\infty^*-\lambda) e^{(t-\tau)\mathcal L_\infty^*}$ is bounded
(see Theorem~5.2 of \cite[Ch. 2]{pazy2012semigroups} together with the discussion
preceding that result).  Therefore
the composition, 
\begin{equation}
\mathcal K^*(\tau) =  \mathcal U^*(\tau,s) \,\mathbf M^*(\tau) 
(\mathcal L_\infty^* - \lambda)^{-1} (\mathcal L_\infty^* - \lambda) 
e^{(t-\tau)\mathcal L_\infty^*},
\end{equation}
is also compact. 

Finally, by Hypothesis~\ref{hyp:PsiEvOp}, the function 
$\mathbf M: (s,t)\to \mathcal B(L^2(\mathbb R,\mathbb C^2))$ is uniformly bounded and strongly
continuous.  Furthermore,  by Theorem~\ref{thm:EvOp}, 
$\mathcal U(\cdot,s)$ is uniformly bounded and
strongly continuous for each $s$. Therefore, $\mathcal K(\tau)$ is strongly continuous, since the composition of uniformly bounded, strongly continuous functions is  strongly continuous.
\qed

\begin{acknowledgements}
We thank the reviewer for their careful reading of the paper. 
\end{acknowledgements}

\bibliographystyle{spmpsci}      
\bibliography{../../BibFiles/PreambleSpaces,../../BibFiles/Stability,../../BibFiles/Lasers}   

\end{document}